\documentclass[letterpaper]{scrartcl}
\usepackage[margin=1.1in]{geometry}
\title{Biography of Paul Erd\H{o}s and discussion of his proof of Bertrand's Postulate}
\author{Meredith M. Paker \\ April 6, 2016 \thanks{This essay was prepared for a History of Mathematics class at the University of Georgia in 2016 and is posted here for others to reference. The essay is not peer-reviewed and has not been updated since 2016. The biographical details rely heavily on other sources especially Schechter's biography of  Erd\H{o}s, \textit{My Brain is Open: The Mathematical Journeys of Paul Erd\H{o}s} (1998). All errors are my own.}
 \date{\vspace{-3.2ex}}}

\usepackage[utf8]{inputenc}
\usepackage[english]{babel}

\usepackage{amsmath}
\usepackage{amsthm}
\usepackage{amssymb}
\usepackage{enumerate}
\usepackage{microtype}
\usepackage{url}
\usepackage{verbatim}
\newcommand{\Ss}{\mathcal{S}}
\newcommand{\Pp}{\mathcal{P}}

\newtheorem{theorem}{Theorem}[section]
\newtheorem{corollary}[theorem]{Corollary} 
\newtheorem{lemma}[theorem]{Lemma}

\begin{document}
\maketitle


\section{Erd\H{o}s' life}

Paul Erd\H{o}s (1913 - 1996) was the most prolific and influential mathematician of the twentieth century and one of the most important mathematicians of all time. In over one thousand papers with hundreds of different coauthors, Erd\H{o}s contributed to dozens of fields, some of which owe their existence to his work \cite{encyclopedia}. To Erd\H{o}s, the meaning of life was ``to prove and conjecture and to keep the SF's score low''\cite[p. 70]{open}. One of many terms and abbreviations he coined, ``SF,'' or ``Supreme Fascist,'' is what he called the imaginary God-like figure that authored ``The Book'' of the most elegant mathematical proofs in the universe and who keeps these proofs initially hidden from mathematicians \cite[pp. 50-57]{hoffman}. By proving everything a mathematician could, he or she could limit the number of points the SF ``won'' in this imaginary tug-of-war \cite[pp. 70-71]{open}. Driven by this life-affirming desire to uncover the most beautiful secrets of math, Erd\H{o}s worked constantly with everyone he met, contributing important results until the day he died and embodying a new philosophy and approach to the subject.

\subsection{Brief Biography}
Paul Erd\H{o}s was born on March 26, 1913, in Budapest, Hungary, to high school math teachers Lajos Erd\H{o}s and Anna Wilhelm \cite{encyclopedia}\cite[p. 24]{open}. His mother's fear of disease kept Erd\H{o}s home from school for much of his early life, which offered him ample time to consider the properties of numbers while learning from her and his governess \cite{encyclopedia}\cite[p. 11]{boy}. Famously, he ``invented'' negative numbers at the age of four, saying that 100 minus 250 was ``150 below zero!'' \cite[p. 29]{open}. He is also noted for impressing his parents' friends as a small child by calculating the exact number of seconds they had been alive \cite[p. 6]{boy}.

\medskip

 When Erd\H{o}s was 10, his father returned from a six-year imprisonment in a Siberian labor camp and began teaching Erd\H{o}s more sophisticated mathematics \cite[p. 43]{open}. During this time, Erd\H{o}s developed a love for prime numbers which was heightened when his father shared with him Euclid's elegant proof that there are infinitely many primes. Erd\H{o}s would begin proper schooling at age 11 in the sixth grade, where he had considerable academic success and continued pursuing deeper mathematical questions on the side \cite[p. 51]{open}.

\medskip

A particularly important influence in Erd\H{o}s' early life was the Hungarian mathematics journal for high schoolers \emph{K\"{o}Mal}. In addition to articles written by prominent mathematicians, the journal contained challenging contest problems. Students from all over Hungary aspired to mail in the neatest solutions to these problems in order to have their picture printed in the magazine \cite[p. 35]{boy}. In 1926, at age 13, Erd\H{o}s' picture was printed in \emph{K\"{o}Mal} along with his future friend Paul Tur\'{a}n after they were the only two respondents who had correctly solved one of the challenge problems \cite[pp. 55-56]{open}.

\medskip

In 1930, Erd\H{o}s began his undergraduate studies at the Hungarian university P\'{a}zm\'{a}ny P\'{e}ter Tudom\'{a}nyegyetem. While at university, Erd\H{o}s contributed some of his best work, including his proof of Bertrand's Postulate, discussed in detail below. He also joined a community of impressive Hungarian mathematicians, beginning his lifelong commitment to collaboration. He and his friends Esther Klein, George Szekeres, and Paul Tur\'{a}n often met at the Statue of Anonymous in Budapest to discuss mathematics \cite[p. 35]{boy}. These discussions led to many important results, including the foundations of Ramsey Theory. After studying at P\'{a}zm\'{a}ny for four years, Erd\H{o}s graduated with a PhD in pure mathematics at the young age of 21 \cite{encyclopedia}.

\medskip

In 1934, he was awarded a postdoctoral fellowship at the University of Manchester, where he continued his research for four years. In 1938, largely due to Hitler's ascension to power in Germany, Erd\H{o}s moved to the United States, where he was offered a one-year position at the Institute for Advanced Study at Princeton \cite{encyclopedia}. When, to his dismay, he was unable to renew his fellowship there, he began to travel the globe, working on questions in numerous fields while collaborating with hundreds of mathematicians from around the world. He held short positions at various universities in the United States such as Penn, Notre Dame, Purdue, Stanford, and Johns Hopkins. He also traveled extensively in dozens of countries including China, Israel, and Australia \cite{encyclopedia}.

\medskip

Throughout these travels, he subsisted primarily on the generosity of his collaborators. He would arrive in a city, often unannounced, and hope to be housed and fed for days or even weeks. Erd\H{o}s had never developed the skills he needed to take care of himself---his kind mother had done his laundry, paid his bills, and cooked him meals for much of his life \cite[p. 23]{boy}. It is even said that when Erd\H{o}s left his mother's care for the first time on a visit to Cambridge, he did not know how to butter the bread he was served at mealtime \cite[p. 92]{open}. Deficiencies like these made Erd\H{o}s a difficult houseguest. In one anecdote that Hungarian mathematician J\'anos Pach recounts, Erd\H{o}s could not figure out how to open a carton of tomato juice. So, after some reflection, he ``decided to get the juice out of the carton by stabbing it with a big knife," \cite[p. 21]{hoffman}, which certainly left a mess! But for all of the annoyances a visit from Erd\H{o}s could bring, his hosts could be certain that they would be inspired and challenged by the work they would do with him. When Erd\H{o}s would rise in the middle of the night, wake up his host, and declare ``My brain is open!" his friends would know that incredible mathematical work was about to occur \cite[p. 24]{boy}.

\medskip

Despite needing the support of his friends and colleagues, Erd\H{o}s proved to be incredibly generous. Though he never had much money or many possessions, if he heard of an aspiring mathematician in need, he would send whatever he had to help. When he was awarded the Wolf Prize, he famously saved only \$720 of his \$50,000 award, using the rest to establish a scholarship at Technion University in Israel in his mother's memory \cite[p. 165]{open}. He also created a series of prizes for anyone who could prove open conjectures that he was interested in, which was an endeavor that his friend Ron Graham managed for him later in his life \cite[pp. 192-196]{open}.

\medskip

Though Erd\H{o}s never married or had children of his own, he loved talking with and entertaining the children of his colleagues. He called small children ``epsilons'' and impressed them by dropping a pill bottle and catching it again before it hit the floor \cite[p. 22, 36]{boy}. In these anecdotes, his pill bottle likely contained amphetamines, which he took daily in addition to drinking large quantities of coffee \cite[p. 196]{open}. Under the influence of these stimulants, Erd\H{o}s often worked on math for twenty hours a day \cite{encyclopedia}, leading to the quip attributed to him that ``a mathematician is a machine for turning coffee into theorems" \cite[p. 15]{open}.

\medskip

His heavy amphetamine use caught up with him in the 1980s when he developed heart troubles and had to undergo a cornea transplant \cite[pp. 241-246]{hoffman}. However, by that time, Erd\H{o}s had been joking about his mortality for years. Before he turned sixty, he would sign his name with the addendum ``PGOM,'' which stood for ``Poor Great Old Man.'' At age sixty he added ``LD'' for ``Living Dead,'' and at age sixty-five he added ``AD'' for ``Archaeological Discovery.'' When he made it to age seventy, he added ``LD'' for ``Legally Dead.'' Finally, his last addition at age seventy-five---``CD'' for ``Counts Dead''---represented the fact that he was dropped from the membership rolls of the Hungarian Academy of Sciences for being too old \cite[p. 41]{open}. On September 20, 1996, he eventually did pass away at age 83 at a conference in Warsaw, Poland. In true Erd\H{o}s fashion, he was contributing mathematical insights up to his final hours \cite{encyclopedia}.

\medskip

Erd\H{o}s' mathematical contributions were recognized throughout his life in a number of prestigious awards. Most notably, in 1984 he was awarded the Wolf Prize, which is one of the highest awards in mathematics. In 1951, he was also awarded the Cole Prize for his work on prime numbers, and he was an honored attendee of the first International Conference on Graph Theory in 1959, which was a field his work helped legitimize \cite{encyclopedia}\cite[pp. 153-159]{open}. Additionally, Erd\H{o}s was inducted into the Hungarian Academy of Sciences in 1956, which led to further fruitful collaborations \cite{encyclopedia}. He was also inducted into the United States National Academy of Sciences in 1979 and into the British Royal Society in 1989. Erd\H{o}s' numerous contributions continue to be recognized today by the prestige associated with having a low ``Erd\H{o}s number,'' which is a measure reflecting the degrees of separation from Erd\H{o}s. His direct collaborators all bear an Erd\H{o}s number of 1, while their collaborators are bestowed an Erd\H{o}s number of 2, and so on.

\subsection{Historical influences}

After the establishment of the Austro-Hungarian empire in 1867, many Jewish people immigrated to Hungary. The newly semi-independent Hungary had recently repealed some anti-Semitic laws, making Hungary a favorable place for Jewish people to live and work \cite[p. 22]{open}. By 1900, Budapest was the sixth-largest city in Europe with a population of over 750,000 \cite[p. 21]{open}. Jewish Hungarians dominated commerce as well as the sciences and the arts, which Schechter speculates is due to the high value they placed on education and the overall strength of the Hungarian education system \cite[p. 23]{open}.

\medskip

Though Erd\H{o}s never strictly adhered to the Jewish faith, he identified culturally as Jewish and was subject to discrimination. He was born in 1913, just before the outbreak of World War I and near the end of Hungary's golden age of intellectual achievement and tolerance. When he was very young, his life was first touched by war when his father was drafted to fight the Russians in World War I. His father then was taken as a prisoner of war and worked for six years in a hard labor camp in Siberia, returning to Hungary in 1920 \cite[p. 42]{open}

\medskip

While his father was gone, Erd\H{o}s' mother tried to shield him from Hungary's political instability. When the dual Austro-Hungarian monarchy dissolved in 1918, Hungary briefly became a Soviet commune. Violence was used to enforce Communist ideals in the ``Red Terror.'' Then, when the Communist regime was overthrown, an even more violent ``White Terror'' began under the new government of Hungary. Because a few prominent Jewish men had been Soviet commissioners, Hungarian Jews began suffering anti-Semitic persecution \cite[p. 42]{open}. Erd\H{o}s once reflected on these circumstances, ``At twelve I knew that eventually I'd have to leave Hungary because I am a Jew" \cite[p. 56]{open}. 

\medskip

After achieving his PhD in Hungary, his parents had hoped he would continue his studies at a prestigious university in Germany. However, due to the deteriorating political situation there, Erd\H{o}s instead accepted a fellowship at the University of Manchester in England  \cite[p. 91]{open}. When Hitler annexed Austria in March 1938 and then threatened to annex the Sudetenland, Erd\H{o}s knew that he needed to leave Hungary. In September of that year, he immigrated to America, leaving many friends and family behind \cite[p. 100]{open}.

\medskip

In addition to forcing his emigration from his homeland, World War II affected the work that Erd\H{o}s did in America. Though he spent his days doing math, his mind was back in Hungary. When Hitler invaded Hungary in 1944 and subsequently murdered almost half a million Jewish Hungarians, Erd\H{o}s was left wondering which of his family and friends might have survived. In 1945 he learned that his beloved mother had survived but that the vast majority of the other friends and family he had left behind perished in the tragedy \cite[p. 132]{open}.

\medskip

When Erd\H{o}s finally returned to Hungary in 1948, a full ten years after he left, he realized that again he could not stay. Hungary was caught in the middle of the Cold War, and the Soviets had just begun a series of gruesome public trials \cite[p. 161]{open}. However, he did not have a second home in the United States either. In 1954, he was investigated under McCarthyist policies for possible Communist connections. Lacking the savvy to say the right things in his interview, he was prevented from reentering the United States until 1963 \cite[pp. 163-167]{open}. The only country that would permit him to enter was Israel because of their constitutional Law of Return for any Jewish people. He accepted this new citizenship and taught at the Technion in Haifa, where he eventually donated most of the money from his Wolf Prize.

\medskip

These mobility restrictions and communication restrictions sometimes made it difficult for  Erd\H{o}s to conduct research with mathematicians in other parts of the world. However, despite the challenges and oppression he faced throughout his life, Erd\H{o}s was still able to contribute a large number of significant mathematical results, which is a testament to his commitment to pursuing math above all else.

\section{Erd\H{o}s' work}

What is most remarkable about Erd\H{o}s' $1,500$ papers written with close to $500$ distinct collaborators is the sheer breadth of fields that Erd\H{o}s touched. In one anecdote from his days at the Institute for Advanced Study at Princeton, Erd\H{o}s was walking through a common room when he heard two famous mathematicians---Witold Hurewicz and Henry Wallman---trying to determine the dimension of the set of the rational points in Hilbert space. Completely unfamiliar with this kind of topology, Erd\H{o}s asked them to explain their problem and then to define the concept of dimension. Within an hour he returned to where Hurewicz and Wallman were still debating whether the answer was $0$ or $\infty$. Erd\H{o}s had proved independently and correctly that the answer was $1$. Upon the publication of the result, his friend Paul Halmos commented that Erd\H{o}s had again made an ``important contribution to a subject that a few months earlier [he] knew nothing about'' \cite[pp. 103-104]{open}.

\medskip

Some of his other important results took months or years to sort out. In college, Erd\H{o}s discovered an elementary proof of Bertrand's Postulate, discussed below, that was so elegant that it was ``straight from The Book" of the SF's perfect proofs. While at university, he also proved Schur's difficult conjecture that the density of abundant numbers is greater than zero \cite{encyclopedia}. With his friends Esther Klein, George Szekeres, and Paul Tur\'{a}n  at the Anonymous statue, he considered a problem in combinatorial geometry called ``The Happy End Problem'' which involved creating convex polygons out of different numbers of points. This led to the Party Problem, which is a foundational problem in Ramsey Theory. Simply stated, the question asks how large a party must be for either $n$ people to all know each other or for $n$ people to all be strangers  \cite[p. 85]{open}. Much later on in his life, Erd\H{o}s continued developing Ramsey Theory and random graphs in important work with Alfred Renyi \cite[pp. 153-156]{open}.

\medskip

Erd\H{o}s' contributions to number theory are particularly impressive. His introduction of probabilistic methods was revolutionary to the field \cite[p. 112]{open}. In one of his greatest but most controversial contributions, he found an elementary proof of the important Prime Number Theorem at the same time as Atle Selberg, for which Selberg was later awarded the Fields Medal \cite[pp. 143-151]{open}. Erd\H{o}s also made significant contributions throughout his life to related fields \cite[p. 97]{open}. For example, he contributed important work on distinct distances in graph theory \cite{encyclopedia} and on inaccessible cardinals in set theory \cite{encyclopedia}. 

\medskip

Erd\H{o}s' wide reach in mathematics was largely due to his willingness to work with anyone who had a good idea, no matter their age or their level of prestige. In one whimsical example, University of Georgia professor Carl Pomerance discovered that the sums of the prime factors of 714 and 715 were equal. These numbers had been in the news frequently during 1974 as Hank Aaron attempted to break Babe Ruth's record of 714 career home runs. Pomerance called consecutive numbers with the same sum of prime factors ``Ruth-Aaron pairs'' and published a short paper with Carol Nelson and David Penney in an obscure journal about the result. One day, he received a call from Erd\H{o}s, who had seen the paper and who had an idea for proving the density of Ruth-Aaron pairs \cite[pp. 179-180]{open}. This began a series of collaborations that resulted in more than twenty coauthored papers, and there are dozens of similar stories \cite{boy}.

\subsection{Proof of Bertrand's Postulate}

In 1845, French mathematician Joseph Bertrand  postulated that for any integer $n > 3$, there exists prime $p$ such that $n < p < 2n-2$ \cite{bertrand}. This is commonly restated in the following slightly weaker form:

\begin{theorem}
\label{bertrand}
Let $n \ge 2$ be an integer. Then there exists prime $p$ such that $n < p < 2n$.
\end{theorem}

Bertrand verified his postulate for $n$ up to three million. In 1850, Russian mathematician Pafnuty Chebychev first proved this postulate analytically \cite{cheby}. More than eighty years later, while just an eighteen-year-old student at P\'{a}zm\'{a}ny P\'{e}ter Tudom\'{a}nyegyetem, Erd\H{o}s independently developed a ``neater" proof of Bertrand's Postulate. Though he later found out that the Indian mathematician Srinivasa Ramanujan had a similar proof buried in his collected papers, independently proving this postulate was Erd\H{o}s' first major achievement \cite[pp. 59-62]{open}. Building on this proof, Erd\H{o}s made a number of important discoveries that earned him his doctorate and helped to establish the legitimacy of elementary methods in number theory research. After the following detailed discussion of his proof, I will discuss the significance of his work more thoroughly.

\medskip

Erd\H{o}s presents his proof of Bertrand's Postulate in his 1932 paper ``Beweis eines Satzes von Tschebyshef," \cite{erdos32} which was published in German in an obscure Hungarian journal. Erd\H{o}s neatly structures this proof into six sections, the first five of which provide the elementary proof of Bertrand's Postulate. Erd\H{o}s assumed for contradiction that there is no prime $p \in (n, 2n)$ and then found a contradiction for large $n$ using lemmas he proved in the first three sections. 

\medskip
Erd\H{o}s' proof is worth exploring because it is similar to today's preferred proof of Bertrand's Postulate and because it is simple enough for an undergraduate to understand while remaining a deep proof. In the following subsections, I carefully unpack the structure of his original proof and the general ideas behind his approach. However, in some sections I neaten his arguments for clarity based on a number of modern sources that have reevaluated his arguments \cite{thebook}\cite{cuttheknot}\cite{galvin}\cite{nielsen}.

\subsection{Erd\H{o}s' proof parts one and two}
In the first two sections of his proof, Erd\H{o}s proves the following lemma of Chebychev:
\begin{lemma}[Chebychev bound] \label{chebybound}
For integers $n \ge 2,$ $\prod_{p\le n}{p} \le 4^n$
\end{lemma}

Erd\H{o}s' approach requires two steps. In the first section he proves that $( \prod_{n < p \le 2n}p ) < 2^{2(n-1)}$. Then, in the second section, he shows that this implies that $( \prod_{10 < p \le n}p ) < 2^{2n}$. This two-step process for proving the Chebychev lemma is less elegant than the induction argument presented below, which is drawn primarily from Nielsen \cite{nielsen}.

\begin{proof}

We can use induction, proving the inductive step in two cases: where $n$ is odd and $n$ is even. The base case $n=2$ is easily verified.

\begin{enumerate}[\upshape (i)]
\item \label{o} For an odd integer $n = 2m + 1,$ let $\displaystyle P = \prod_{p\le n}{p} \le 4^n$. For the following (even) integer $n+1 = 2(m+1), \displaystyle \prod_{p \le n+1}{p}$ also equals $P$ because $n+1$ is not prime. Thus $\displaystyle \prod_{p \le n+1}{p} \le 4^{n+1}$.
\item \label{e} 
For an even integer $n = 2m$, we assume that $\displaystyle \prod_{p\le k}{p} \le 4^k$ for all $2 < k \le n$. We aim to show that for the following (odd) integer $n+1 = 2m+1$, $\displaystyle \prod_{p\le n+1}{p} \le 4^{n+1}$. First note that
\[
\left ( \displaystyle \prod_{p\le 2m+1}{p}\right ) = \left ( \displaystyle \prod_{p\le m+1}{p} \right ) \left (\prod_{m+1 < p \le 2m+1}{p} \right ).
\]

$\left ( \displaystyle \prod_{p\le m+1}{p} \right ) \le 4^{m+1}$ by the inductive hypothesis, since $m+1 \le n$.

$\left ( \displaystyle \prod_{m+1 < p \le 2m+1}{p} \right )$ divides $\displaystyle\binom{2m+1}{m}$ because the 
numerator of $\displaystyle \binom{2m+1}{m}$ = $\displaystyle \frac{(2m+1)(2m)...(m+2)}{m(m-1)...(1)}$ contains all integers between $m+2$ and $2m+1$, including all primes in this range, and the denominator contains none.

Therefore, $\left ( \displaystyle \prod_{m+1 < p \le 2m+1}{p} \right ) \le \displaystyle \binom{2m+1}{m}$.

Consider the binomial expansion
\[
(1+1)^{2m+1} = \binom{2m+1}{0}\cdot 1^{2m+1}\cdot(1) + \ldots + \binom{2m+1}{m}\cdot 1^{m+1}\cdot(1^m) + \binom{2m+1}{m+1}\cdot 1^{m}\cdot(1^{m+1})\ldots
\]

Note that 
\[
\binom{2m+1}{m}\cdot 1^{m+1}\cdot(1^m) =  \binom{2m+1}{m+1}\cdot 1^{m}\cdot(1^{m+1}).
\]

Each of these identical terms can contribute at most half of the total value of the binomial, so we have
\[
\binom{2m+1}{m}\cdot 1^{m+1}\cdot(1^m) \le \frac{(1+1)^{2m+1}}{2}
\]
\[
\binom{2m+1}{m} \le \frac{2^{2m+1}}{2} = 2^{2m}.
\]

Therefore, $\left ( \displaystyle \prod_{m+1 < p \le 2m+1}{p} \right ) \le \displaystyle\binom{2m+1}{m} \le 4^{m}.$

Finally,
\[
\left ( \displaystyle \prod_{p\le 2m+1}{p}\right ) \le \left ( 4^{m+1} \right ) \left ( 4^{m} \right ) = 4^{2m+1}.
\]
\end{enumerate}    
\end{proof}

\subsection{Erd\H{o}s' proof part three}
In the third section of his proof, Erd\H{o}s proves three lemmas using the Lengendre formulas for prime factorization of factorials and a fourth lower-bound lemma. 
\begin{lemma} \label{twon}
For prime $p$, define $b(n,p)$ to be the largest number $b$  such that $p^b$ divides $ \binom{2n}{n}$. Then $p^b \le 2n$ for all $p$.
\end{lemma}

\begin{proof}
Note that the power $a$ of $p$ in the prime factorization of $n!$ is given by
\[
a(n,p) = \displaystyle \sum_{i=1}^{\infty} \left \lfloor { \frac{n}{p^i}} \right \rfloor,
\]
and that the corresponding power of $p$ in the prime factorization of $\frac{1}{n!}$ is given by $-a(n,p)$.

\medskip

Recall that  $ \displaystyle\binom{2n}{n} = \displaystyle \frac{2n!}{(n!)(n!)} = \left ( 2n!\right ) \left ( \displaystyle\frac{1}{n!} \right ) \left (\displaystyle \frac{1}{n!} \right )$. Applying the above formula, we have
\[
b(n,p) = \left ( \displaystyle \sum_{i=1}^{\infty} \left \lfloor { \frac{2n}{p^i}} \right \rfloor  \displaystyle - \sum_{i=1}^{\infty} \left \lfloor{ \frac{n}{p^i}} \right \rfloor \displaystyle - \sum_{i=1}^{\infty} \left \lfloor{ \frac{n}{p^i}} \right \rfloor \right ) \mbox{,  which gives}
\]
\[
b(n,p) = \displaystyle  \sum_{i=1}^{\infty}\left ( \left \lfloor { \frac{2n}{p^i}} \right \rfloor - \displaystyle 2  \left \lfloor{ \frac{n}{p^i}} \right \rfloor \right ).
\]

Next, notice that if $\displaystyle  \left ( \frac{n}{p^i} - \left\lfloor \frac{n}{p^i} \right \rfloor \right ) \ge \frac{1}{2}$ then $\displaystyle \left \lfloor  \frac{2n}{p^i} \right \rfloor =  2  \left \lfloor \frac{n}{p^i} \right \rfloor + 1$. In this case, the summand
 $\displaystyle \left ( \left \lfloor  \frac{2n}{p^i} \right \rfloor -  2  \left \lfloor \frac{n}{p^i} \right \rfloor \right )$ equals $1$. On the other hand, if $\displaystyle  \left ( \frac{n}{p^i} - \left\lfloor \frac{n}{p^i} \right \rfloor \right ) < \frac{1}{2}$ then $\displaystyle \left \lfloor  \frac{2n}{p^i} \right \rfloor =  2  \left \lfloor \frac{n}{p^i} \right \rfloor$. In this case, the summand equals 0. Therefore,
 \begin{equation}\label{eq:in01}
 \displaystyle \left ( \left \lfloor  \frac{2n}{p^i} \right \rfloor -  2  \left \lfloor \frac{n}{p^i} \right \rfloor \right ) \in \{0, 1\}.
 \end{equation}
 Now compare the growth of the variables in the summand:
 \[ \mbox{Suppose } p^i > 2n.
 \]
 This is the case if and only if  $i > \displaystyle \frac{\log{2n}}{\log{p}},$ which gives
 \[ i > \log_p{2n}.
 \]
This implies that when $i > \log_p{2n}$, the corresponding summand contributes $0$ to $b(n,p)$. From equation (\ref{eq:in01}), when $i \le \log_p{2n}$, the corresponding summand of $b$ equals $0$ or $1$. Therefore, the maximum value $b$ can obtain is $\log_p{2n}.$ After some rearranging using the definition of a logarithm, we have the result:
\[
p^b \le 2n.
\]
\end{proof}

\begin{lemma} \label{root2n}
$b(n,p)$ is less than or equal to 1 for all primes $p > \sqrt{2n}$
\end{lemma}

\begin{proof}
Recall from Lemma \ref{twon} that
\[
b(n,p) = \displaystyle  \sum_{i=1}^{\infty}\left ( \left \lfloor { \frac{2n}{p^i}} \right \rfloor - \displaystyle 2  \left \lfloor{ \frac{n}{p^i}} \right \rfloor \right )
\]

Note that $\frac{2n}{\sqrt{2n}^2}$ = 1, so for $p > \sqrt{2n}$ and $i \ge 2$, $\frac{2n}{p^i} < 1$. Therefore, $\left \lfloor \frac{2n}{p^i} \right \rfloor = 0$ for $p > \sqrt{2n}$ and all $i \ge 2$. $n \le 2n$ implies that the same holds true for $\left \lfloor \frac{n}{p^i} \right \rfloor$. Thus we have
\[
b(n,p) =  \left \lfloor { \frac{2n}{p^1}} \right \rfloor - 2  \left \lfloor{ \frac{n}{p^1}} \right \rfloor \le 1 \mbox{ for } p > \sqrt{2n}.
\]

\end{proof}

\begin{lemma} \label{twothirds}
$b(n,p)$ equals 0 for all primes $p$ with $\frac23 n < p \le n$ and $n \ge 2$
\end{lemma}

\begin{proof}
For $n>\frac92$, note that $\frac23n > \sqrt{2n}$. Then by Lemma \ref{root2n}, we know that
\[
b(n,p) =  \left \lfloor { \frac{2n}{p}} \right \rfloor - 2  \left \lfloor{ \frac{n}{p}} \right \rfloor \le 1.
\]
For $p > \frac23n$ and $p \le n$,
\[
  \frac{2n}{n} \le \frac{2n}{p} < \frac{2n}{\frac23 n}
\]
So, $\frac{2n}{p} \in [2, 3)$. Therefore, $\left \lfloor { \frac{2n}{p}} \right \rfloor = 2$.

Likewise, for $p > \frac23n$ and $p \le n$,
\[
  \frac{n}{n} \le \frac{n}{p} < \frac{n}{\frac23 n}
\]
So, $\frac{n}{p} \in [1, \frac32)$. Therefore, $\left \lfloor { \frac{n}{p}} \right \rfloor = 1$ and $2  \left \lfloor{ \frac{n}{p}} \right \rfloor = 2$.

Finally, for $p > \frac23n$ and $p \le n$, where $n > \frac92$, we have
\[
b(n,p) =  \left \lfloor { \frac{2n}{p}} \right \rfloor - 2  \left \lfloor{ \frac{n}{p}} \right \rfloor = 2-2 = 0.
\]
The cases $n<\frac92$, that is, $n$ equals $2$, $3$, or $4$, are easily verified.
\end{proof}

\begin{flushleft}\emph{Comment.} Note that
\end{flushleft}
\begin{equation*}
\binom{2n}{n} = \left ( \prod_{p \le \sqrt{2n}}{p^b} \right ) \left ( \prod_{\sqrt{2n} < p \le \frac23 n}{p^b} \right ) \left ( \prod_{\frac23 n < p \le n}{p^b} \right ) \left ( \prod_{n < p \le 2n}{p^b} \right )
\end{equation*}
By Lemma \ref{twon},  $\prod_{p \le \sqrt{2n}}{p^b} \le (2n)^{\sqrt{2n}}$ and by Lemma \ref{twothirds}, $\prod_{\frac23 n < p \le n}{p^b} = 1$. This gives
\begin{equation}\label{eq:binom}
\binom{2n}{n} \le \left ( (2n)^{\sqrt{2n}} \right ) \left ( \prod_{\sqrt{2n} < p \le \frac23 n}{p^b} \right ) \left ( \prod_{n < p \le 2n}{p^b} \right )
\end{equation}

\begin{lemma}[Lower bound lemma] \label{lowerbound}
$2n\binom{2n}{n} > 2^{2n}$ for $n \ge 2$
\end{lemma}

\begin{proof}
Consider the binomial expansion with 2n terms:
\[
(1+1)^{2n} = 2 + \binom{2n}{1}\cdot 1^{2n-1}\cdot 1^{1} + \dots +  \binom{2n}{2n-1} \cdot 1^1 \cdot 1^{2n-1}
\]
Note that $\binom{2n}{n}$ is the middle term of this expansion and thus the biggest. Replacing each term in this binomial expansion with $ \binom{2n}{n}$ increases the value of the sum to $2n\binom{2n}{n}$. So we have $2n\binom{2n}{n} > 2^{2n}$.
\end{proof}

\subsection{Erd\H{o}s' proof part four}
In the fourth section of his proof, Erd\H{o}s completes the argument for $n \ge 4000$ by assuming there are no primes $p \in (n, 2n)$ and seeking a contradiction.

Recall from equation (\ref{eq:binom}) that
\[
\binom{2n}{n} \le \left ( (2n)^{\sqrt{2n}} \right ) \left ( \prod_{\sqrt{2n} < p \le \frac23 n}{p^b} \right ) \left ( \prod_{n < p \le 2n}{p^b} \right ).
\]
Applying Lemmas \ref{chebybound} and \ref{root2n}, we have $\prod_{\sqrt{2n} < p \le \frac23 n}{p^b} \le 4^{\frac23 n}$. By assumption, there are no primes between $n$ and $2n$. This gives
\[
\binom{2n}{n} \le \left ( (2n)^{\sqrt{2n}} \right ) \left( 4^{\frac23 n} \right).
\]
We can multiply each side by $2n$ to get
\[
2n \cdot \binom{2n}{n} \le 2n \cdot \left ( (2n)^{\sqrt{2n}} \right ) \left( 4^{\frac23 n} \right).
\]
By Lemma \ref{lowerbound}, $2n\binom{2n}{n} > 2^{2n}$, which gives a statement that will be contradicted for large $n$:
\[
 2^{2n} < 2n \cdot \binom{2n}{n} \le 2n \cdot \left ( (2n)^{\sqrt{2n}} \right ) \left( 4^{\frac23 n} \right).
\]

\begin{lemma}
For $n \ge 4000$, $ 2^{2n} \ge  2n \cdot \left ( (2n)^{\sqrt{2n}} \right ) \left( 4^{\frac23 n} \right)$
\end{lemma}

\begin{proof}
Consider the limit:
\[
\lim_{n \to \infty} \frac{2^{2n}}{2n \cdot (2n)^{\sqrt{2n}} \cdot 4^{\frac23 n}} = \lim_{n \to \infty} \frac{2^{2n}}{(2n)^{1 + \sqrt{2n}} \cdot 2^{2\cdot \frac23 n}} = \lim_{n \to \infty} \frac{2^{\frac23 n}}{(2n)^{1 + \sqrt{2n}}}
\]

Taking the log of the numerator and denominator and applying L'H\^{o}pital's Rule gives
\[
\lim_{n \to \infty} \frac{2^{2n}}{2n \cdot (2n)^{\sqrt{2n}} \cdot 4^{\frac23 n}} = \lim_{n \to \infty} \frac{\frac23 \cdot \log{2}}{\frac{1+\sqrt{2n}}{2n} + \frac{\log{2n}}{2 \sqrt{2n}}} = \infty.
\]
Therefore, for sufficiently large $n$, 
\[
\frac{2^{2n}}{2n \cdot (2n)^{\sqrt{2n}} \cdot 4^{\frac23 n}} > 1.
\]
This implies that $ 2^{2n} \ge  2n \cdot \left ( (2n)^{\sqrt{2n}} \right ) \left( 4^{\frac23 n} \right)$. Here $n \ge 4000$ is sufficiently large.
\end{proof}
\subsection{Erd\H{o}s' proof part five}
In the fifth section, Erd\H{o}s demonstrates that Bertrand's Postulate holds for $n<4000$. He simply produces the following list of primes where each consecutive term is less than twice the previous term:
\[
2, 3, 5, 7, 13, 23, 43, 83, 163, 317, 631, 1259, 2503, 4001
\]
This proves that for $n < 4000$, there always exists prime $p$ such that $p \in (n, 2n)$, which completes the proof for all $n$.

\subsection{Significance of the problem}

Erd\H{o}s' proof of Bertrand's Postulate was important mathematically, personally, and philosophically. Mathematically, it led to even more significant results that formed the bulk of the PhD dissertation Erd\H{o}s completed after just four total years of university schooling \cite[p. 62]{open}. After proving Bertrand's Postulate, Erd\H{o}s extended his work to prove that for $n>7$, there are at least two primes $p_1, p_2 \in (n, 2n).$ Furthermore, $p_1 \equiv 3 \bmod{4}$ and $p_2 \equiv 1 \bmod{4}$ without loss of generality \cite[p. 62]{open}.

\medskip

Personally, this work was important because it began Erd\H{o}s' long tradition of corresponding with and collaborating with mathematicians all over the world. Schechter names a few of these early penpals---Louis Mordell, Richard Rado, and Harold Davenport in England, as well as Issai Schur in Berlin \cite[p. 62]{open}. As he gained renown among mathematicians throughout the world during this time, Erd\H{o}s earned the nickname ``the Magician from Budapest"  \cite[pp. 62-63]{open}.

\medskip

But Erd\H{o}s' proof of Bertrand's Postulate was most significant philosophically. Though it was just one of his early proofs, Erd\H{o}s had used elementary means to prove a substantive result in an elegant way. This proof was his first proof that was straight from the SF's ``Book," affirming his life's mission to discover other beautiful proofs in ``The Book.'' Further, this proof has continued to encourage mathematicians to consider elementary number theory and to seek elegance in all of their work.

\providecommand{\bysame}{\leavevmode\hbox to3em{\hrulefill}\thinspace}
\providecommand{\MR}{\relax\ifhmode\unskip\space\fi MR }
\providecommand{\MRhref}[2]{%
  \href{http://www.ams.org/mathscinet-getitem?mr=#1}{#2}
}
\providecommand{\href}[2]{#2}

\end{document}